\documentclass[12pt]{amsart}

\usepackage{amssymb}

\usepackage{enumerate}

\usepackage{hyperref}

\usepackage{graphicx}

\makeatletter
\@namedef{subjclassname@2010}{%
  \textup{2010} Mathematics Subject Classification}
\makeatother

\theoremstyle{plain}
\newtheorem{conjecture}{Conjecture}
\newtheorem{theorem}{Theorem}
\newtheorem{lemma}{Lemma}

\theoremstyle{definition}

\theoremstyle{definition}

\numberwithin{equation}{section}

\begin{document}


\baselineskip=17pt



\title[On Moebius orthogonality for subshifts of finite type]{On Moebius orthogonality for subshifts of finite type with positive topological entropy. }

\author[D. Karagulyan]{D. Karagulyan}
\address{Department of Mathematics, Royal Institute of Technology, S-100 44 Stockholm, Sweden}
\email{davitk@kth.se}

\date{}

\begin{abstract}
In this note we prove that Moebius orthogonality does not hold for subshifts of finite type with positive topological entropy. This, in particular, shows that all $C^{1+\alpha}$ surface diffeomorphisms  with positive entropy correlate with the Moebius function.
\end{abstract}

\subjclass[2010]{Primary 37B10; Secondary 37A35, 11Y35}

\keywords{Subshifts of finite type, entropy, M{\"o}bius orthogonality}

\maketitle

\section{Introduction}
Let $\mu$ denote the M{\"o}bius function, i.e.
$$
\mu(n) = \begin{cases} 
(-1)^k & \mbox{ if } n=p_1 p_2 \cdots p_k \mbox{ for distinct primes } p_i, \\
0& \mbox{ otherwise}.\\
\end{cases}
$$
In \cite{Sa1}, \cite{Sar} Sarnak introduced the following related conjecture. Recall
that a topological dynamical system $(Y,T)$ is a compact metric
space $Y$ with a homeomorphism $T : Y \rightarrow Y$ , and the topological entropy
$h(Y,T)$ of such a system is defined as
$$
h(Y,T)=\lim_{\epsilon\rightarrow 0}\lim_{n\rightarrow \infty}\frac{1}{n}\log N(\epsilon,n),
$$
where $N(\epsilon, n)$ is the largest number of $\epsilon$-separated points in $Y$ using the metric $d_n:Y\times Y \rightarrow (0,\infty)$ defined by 
$$
d_n(x,y)=\max_{0\leq i\leq n}d(T^ix,T^iy).
$$
A sequence $f:\mathbb{Z} \rightarrow \mathbb{C}$ is said to be deterministic if it is of the form
$$
f(n)=F(T^nx),
$$
for all $n$ and some topological dynamical system $(Y,T)$ with zero topological entropy $h(Y,T)=0$, a base  point $x \in Y$, and a continuous function $F:Y \rightarrow \mathbb{C}$.
\begin{conjecture}\label{con1}(P. Sarnak)
Let $f:\mathbb{N}\rightarrow \mathbb{C}$ be a deterministic  sequence. Then 
\begin{equation} \label{f0}
S_n(T(x),f) = \frac{1}{n} \sum_{k=1}^n \mu(k)f(k)=o(1).
\end{equation}
\end{conjecture}

In this case we also say, that the Moebius function does not correlate or is orthogonal to the sequence $f(k)$.

The conjecture is known to be true for several dynamical systems. For a Kronecker flow (that is
a translation on a compact abelian group) it is proved in \cite{V} and \cite{davenport}, while when
$(X, f)$ is a translation on a compact nilmanifold it is proved in \cite{G-T}. In \cite{B-S-Z} it is established also for horocycle flows. For orientation preserving circle-homeomorphisms and continuous interval maps of zero entropy it is proved in \cite{D}. For other references see (\cite{A-L}, \cite{AKL}, \cite{K-L}). 

In this paper we study the opposite direction of the conjecture. That is, we are interested in systems with positive entropy and their correlation properties with the Moebius function. Peter Sarnak, in his famous exposition \cite{Sar}, mentions that for a given sequence $\epsilon(n)$ satisfying certain conditions such as $\mu(n)$, one can construct a positive entropy flow orthogonal to $\epsilon(n)$ (he attributes this to Bourgain - private communication).
But this example has never been published.
Assuming Bourgains claim, it becomes of interest to construct positive entropy system which does not correlate with the Moebius function. In this context the subshifts of finite type is a natural class to examine. We mention, that the conjecture in the opposite direction has previously been considered in \cite{AKL} and \cite{D-K}. The authors construct examples of non-regular Toeplitz sequences for which the orthogonality to the Moebius function does not hold. We point out that for any measure preserving dynamical system $(X, \mathcal{B}, \nu, T)$ Sarnak's conjecture holds almost surely with respect to $\nu$ (see \cite{Sa1}). Hence the orthogonality to the Moebius function may fail only on a pathological set. Another motivation for studying subshifts of finite type comes from Katok's famous horseshoe theorem which states, that for any $C^{1+ \varepsilon}$ smooth surface diffeomorphisms with positive entropy, there is a compact invariant set $\Lambda$ such that the restriction of the map on $\Lambda$ is topologically conjugate to a subshift of finite type with positive topological entropy. So this shows, that all sufficiently smooth surface diffeomorphisms with positive entropy correlate with the Moebius function. Horseshoes also emerge in many other systems with positive entropy, such as unimodal maps, H\'{e}non maps etc.

\section{Statement and proof of the main theorem}
We are going to prove the following theorem.

\begin{theorem}
For any subshift $(T,\Sigma_{A}^{+})$ of finite type with positive topological entropy there exist a sequence $z \in \Sigma_{A}^{+}$ and a continuous test function $\phi$, for which
$$
\lim_{n \rightarrow \infty}S_n(T(z),\phi)\neq 0.
$$
\end{theorem}

First we recall the definition of a subshift of finite type. Let $V$ be a finite set of $n$ symbols and $A$ be an $n\times n$ adjacency matrix with entries in $\{0,1\}$. Define
$$\Sigma_{A}^{+} = \left\{ (x_0,x_1,\ldots):
x_j \in V,  A_{x_{j}x_{j+1}}=1, j\in\mathbb{N} \right\}.$$
The shift operator $T$ maps a sequence in the one-sided shift to another by shifting all symbols to the left, i.e.
$$
\displaystyle(T(x))_{j}=x_{j+1}.
$$
The topological entropy of a subshift of finite type can be computed by computing the number of different admissible words of length up to $n$, i.e. if 
$$
B_n=\{(v_0,\cdots, v_{n-1}):v_j=x_j \hbox{ for } 0 \leq j < n \hbox{ for some } x \in \Sigma_{A}^{+}),
$$
then the topological entropy of $T$ equals
\begin{equation}\label{ent}
h(T)=\limsup_{n \rightarrow \infty}\frac{\log \#B_n}{n}.
\end{equation}

We say, that the sequence of symbols $X=\{x_k\}_{k=0}^{n}$ from $V$ is a word in $\Sigma_{A}^{+}$, if we have $A_{x_i,x_{i+1}}=1$ for all $i=0,..,n-1$. For the word $X$, the number of its elements will be denoted by $|X|$.
We say, that we have an admissible loop at $v \in V$, if there exists a finite sequence of elements $\{x_k\}_{k=0}^{n}, n \geq 1$ from $V$ such, that
$x_0=v,x_n=v$ and $x_k \neq v$ for $0<k<n$ and $A_{x_i,x_{i+1}}=1$ for all $i=0,...,n-1$.

The main property of the M{\"o}bius function, which will be used in the proof is the following well known fact (see e.g. \cite{M}, \cite{C-S})
\begin{equation} \label{f1}
\lim_{N \rightarrow \infty}\frac{1}{N}\sum_{k=1}^{N} \mu^2(k)=\frac{6}{\pi^2}.
\end{equation}
The following lemma is an easy consequence of ~\eqref{f1}. It can also be obtained from Mirsky's theorem on the patterns of arithmetic progressions in square free numbers (\cite{M})

\begin{lemma}\label{L1}
For any $M \in \mathbb{N}$ there exist an integer $0\leq s < M$, for which
\begin{equation} \label{f2}
\limsup_{N \rightarrow \infty}\frac{1}{N}\sum_{1 \leq k \leq N \atop k \equiv s \pmod{M}} \mu^2(k)>0.
\end{equation}
\end{lemma}
\begin{proof}
For this it is enough to note that
$$
\frac{1}{N}\sum_{k=1}^{N} \mu^2(k) = \sum_{l=0}^{M-1} \frac{1}{N}\sum_{1 \leq k \leq N \atop k \equiv l \pmod{M}} \mu^2(k),
$$
and since the sum in the left hand side does not converge to 0, then the same must hold at least for one of the $M$ sums on the right hand side.
\end{proof}

Let us return to subshifts of finite type with positive entropy. Note, that for some $v \in V$ there are at least two different admissible loops at $v$. Otherwise the space $\Sigma_{A}^{+}$ will consist of only periodic orbits and from \eqref{ent} it will follow, that the topological entropy of $T$ is zero. Let $\gamma_1=\{x_1,x_2, .., x_n\}$ and $\gamma_2=\{y_1,y_2, .., y_m\}$ be the two loops. Define also the words $\gamma'_1=\{x_1,x_2, .., x_{n-1}\}$ and $\gamma'_2=\{y_1,y_2, .., y_{m-1}\}$.

\begin{lemma}\label{L2}
For any subshift of finite type $(T,\Sigma_{A}^{+})$ with positive topological entropy there exists a positive integer $l \in \mathbb{N}$, such that for any integer $s$, with $0 \leq s < l$, there exists an element $z_s \in \Sigma_{A}^{+}$ and a test function $\phi\in C(\Sigma_{A}^{+})$ with the property, that for any $n \in \mathbb{N}$ the following holds
\begin{equation} \label{f5}
\phi(T^{n}(z_s)) =
\begin{cases} 
\mu(n) & :\mbox{ if } n \equiv s \pmod{l}  \mbox{ and } \mu(n) \neq 0 \\
0 & :\mbox{ otherwise.} \\
\end{cases} 
\end{equation}
\end{lemma}

\begin{proof}

We say that two finite words of equal length $a$
and $b$ have the recognizability property if in any concatenation of these two words
($aa$, $ab$, $bb$ or $ba$) neither $a$ nor $b$ have any “additional” occurrences (not shown in
the writing of the concatenated blocks). For example, it is very easy to check that
$a = 00110$ and $b = 01010$ have the recognizability property. Following the patters of $0$:s and $1$:s in $a$ and $b$, we construct the following two words,
$$
x=\gamma'_1\gamma'_1\gamma'_2\gamma'_2\gamma'_1,
$$
and
$$
y=\gamma'_1\gamma'_2\gamma'_1\gamma'_2\gamma'_1.
$$
Notice, that
$$
|x| = |y| = 3|\gamma'_1| + 2|\gamma'_2|,
$$
and let
$$
l = 3|\gamma'_1| + 2|\gamma'_2|.
$$
Since $\gamma'_1$ and $\gamma'_2$ both start with the symbol $v$ and $v$ does not occur inside these blocks, and the blocks $x$, $y$ are built of $\gamma'_1$ and $\gamma'_2$
following the $a$, $b$ patterns, and
have the same length, it is obvious that $x$, $y$ also have the recognizability property. One can also check, that any word obtained through concatenation of $x$ and $y$ will be admissible in $\Sigma_{A}^{+}$.
We now define a sequence $z_s=\tilde{z}z_0z_1...z_k...$, where $\tilde{z}$ is the word consisting of the last $s$ symbols of either $x$ or $y$, and $\{z_k\}_{k=0}^{\infty}$ is chosen as follows
\begin{equation} \label{f7}
z_k = \begin{cases} 
x & \mbox{ if } \mu(kl+s)=-1, \\
y & \mbox{ otherwise}.\\
\end{cases}
\end{equation}
Clearly $z_s$ is admissible in $\Sigma_{A}^{+}$.
We also define a test function $\phi$ in the following way
\begin{equation}\label{f8}
\phi(u)=\begin{cases} 
1 & \mbox{ if } u \mbox{ begins with the word } x, \\
-1 & \mbox{ if } u \mbox{ begins with the word } y, \\
0 & \mbox{ otherwise}, \\
\end{cases}
\end{equation}
or alternatively
$$
\phi(u)=\chi_{x}(u)-\chi_{y}(u),
$$
where $\chi_{\gamma}(u)$ is the characteristic function of the word $\gamma$. It is easy to check, that $\chi_{\gamma}(u)$ is a continuous function and hence so is $\phi(u)$.

Since the words $x$ and $y$ satisfy the recognizability property, then it is clear, that in the sequence $z_s$ they may appear only at the positions $n \equiv s \pmod{l}$. But from the construction of the test function $\phi$ and the choice of the subwords $\{z_k\}$ in $z_s$ it follows, that $\phi(T^{n}(z_s))=\mu(n)$, if $\mu(n)\neq 0$ and $n \equiv s \pmod{l}$. Hence the lemma is proved.

\end{proof}

\begin{proof} Using Lemma ~\ref{L2} we can find a positive integer $l$ such that for any integer $s$, with $0 \leq s < l$, there exists an element $z_s \in \Sigma_{A}^{+}$ and a continuous test function $\phi\in C(\Sigma_{A}^{+})$ such that ~\eqref{f5} holds. 
Now we use Lemma 1 with $M=l$ to find an integer $s$ for which ~\eqref{f2} holds. 
According to ~\eqref{f5}
$$
\sum_{n=1}^{N}\mu(n)\phi(T^n(z_s))= \sum_{1 \leq k \leq N \atop k \equiv s \pmod{l}} \mu^2(n).
$$
But from the choice of $s$
$$
\limsup_{N \rightarrow \infty}\frac{1}{N}\sum_{1 \leq k \leq N \atop k \equiv s \pmod{l}} \mu^2(n)>0.
$$
Hence
$$
\limsup_{N \rightarrow \infty}\frac{1}{N}\sum_{n=1}^{N} \mu(n)\phi(T^n(z_s))=\limsup_{N \rightarrow \infty}\frac{1}{N}\sum_{1 \leq k \leq N \atop k \equiv s \pmod{l}} \mu^2(n)>0,
$$
which finishes the proof of the theorem.

\end{proof}

One can also show, that the entropy of $T$ is positive.
By definition $|x|=|y|=l$. As we know all concatenations of the words $x$ and $y$ are admissible, hence if we consider all admissible words of length $\{nl, n=1,2 \cdots \}$, then for the number of different words of length $nl$ we will have
\begin{equation}\label{v1}
\frac{\log\#B_{nl}}{nl} \geq \frac{\log 2^{n}}{nl}=\frac{\log 2}{l}=\log2^{1/l}.
\end{equation}

Hence for the entropy of $T$, we have $h(T)\geq \log2^{1/l}$.
Now, if we consider the subshift of finite type, which consists of exactly two loops $|\gamma'_1|=|\gamma'_2|=l$ of length $l$, then in \eqref{v1} we will have equality. Hence we see, that the entropy of $T$ can be made arbitrarily small.

\subsection*{Acknowledgements}
The author would like to thank Mariusz Lema{\'n}czyk for proposing the problem in the opposite direction to Sarnak's conjecture, to Michael Benedicks for his remark about Katok's horseshoe theorem and his guidance and many valuable suggestions, Ana Rodrigues, El Houcein El Abdalaoui for useful comments about the manuscript. I would also like to thank Joanna Ku\l{}aga-Przymus for her remark concerning Lemma 1 and also the anonymous referee for carefully reading the manuscript and for giving such constructive comments, which substantially helped improving the quality of the paper.


\begin{thebibliography}{HD}






\normalsize
\baselineskip=17pt

\bibitem{A-L}
H. El Abdalaoui, M. Lema´nczyk, T. de la Rue, On spectral disjointness
of powers for rank-one transformations and M¨obius orthogonality, J. Functional
Analysis 266 (2014), 284-317.
\bibitem{AKL}
 H. El Abdalaoui, S. Kasjan and M. Lema´nczyk, 0-1 sequences of the Thue-Morse type and
Sarnaks conjecture, Proc. Amer. Math. Soc,
Volume 144, Number 1, January 2016, Pages 161-176.
\bibitem{B-S-Z} 
J. Bourgain, P. Sarnak, T. Ziegler, Disjointness of Moebius from horocycle flows. From Fourier analysis and number theory to radon transforms and geometry, 67-83, Dev. Math., 28, Springer, New York, 2013.
\bibitem{C-S}
F. Cellarosi and Ya. G. Sinai, Ergodic properties of square-free numbers, J. Eur. Math. Soc.
15 (2013), 1343-1374.
\bibitem{D-K}
T. Downarowicz, S. Kasjan, Odometers and Toeplitz subshifts revisited in
the context of Sarnak’s conjecture, Studia Mathematica 229(1).
\bibitem{davenport}
H. Davenport, On some infinite series involving arithmetical functions. II, Quart. J. Math.,
Oxford Ser. 8,(1937), 313-320. 
\bibitem{G-T}
B. Green, T. Tao, The Moebius Function is strongly orthogonal to nilsequences, to
Ann. of Math. (7) 175, 541-566 (2012).
\bibitem{D}
D. Karagulyan, On Mobius orthogonality for interval maps of zero entropy
and orientation-preserving circle homeomorphisms. Ark. Mat., Volume 53, Issue 2, pp 317-327.
\bibitem{K-L}
J. Kulaga-Przymus, M. Lemanczyk, The Mobius function and continuous extensions of rotations  
Monatsh. Math., 178, 553-582, 2015.
\bibitem{M}
Mirsky, L. Arithmetical pattern problems relating to divisibility by rth powers. Proc. London Math. Soc. (2) 50, (1949). 497–508. (Reviewer: W. H. Simons) 10.0X
\bibitem{Sa1}
P. Sarnak, Three lectures on the Mobius Function randomness and dynamics,
publications.ias.edu: \url{http://www.math.ias.edu/files/wam/2011/PSMobius.pdf}.
\bibitem{Sar} P. Sarnak, Mobius randomness and dynamics, Not. S. Afr. Math. Soc. 43
(2012), No. 2, 89–97.
\bibitem{V}
I. M. Vinogradov, Some theorems concerning the theory of primes, Math. Sb. N. S., 2 (1937), 179-195.

\end{thebibliography}
\end{document}